\documentclass{panm}

\usepackage[svgnames]{xcolor}
\usepackage{graphicx}

\usepackage{amsthm}
\usepackage{amsmath}
\usepackage{amssymb}
\usepackage{mathrsfs}
\usepackage{algorithmic}
\usepackage[ruled]{algorithm}
\usepackage{hyperref} 

\newtheorem{theorem}{Theorem}
\newtheorem{lemma}[theorem]{Lemma}
\newtheorem{corollary}{Corollary}

\renewcommand{\vec}{\mathrm{vec}}
\renewcommand{\sp}{\mathrm{sp}}
\newcommand{\rank}{\mathrm{rank}}
\newcommand{\spn}{\mathrm{span}}
\newcommand{\diag}{\mathrm{diag}}

\title{On bounds for rank growth of iterands in conjugate gradients for Lyapunov equation}

\author{Jana Lungov\'a $^1$, Martin Ple\v singer$^1$\\
\vskip 2mm {\small
$^1$
Department of Mathematics, 
Technical University of Liberec\\
Studentsk\'a 1402/2, 461 17 Liberec 1, Czech republic\\
jana.lungova@tul.cz, martin.plesinger@tul.cz \\
}
}

\abstract{
We focus on solving the Lyapunov equation $AX + XA^T = F$, where $A$, $X$ and $F$ are 
square matrices, $A$ is symmetric positive definite (SPD) and sparse, and $F$ is symmetric 
and of a low-rank. The solution $X$ is then also symmetric and in general dense, but
it can be approximated by a low-rank matrix. If $n$ is large, $X$ cannot be computed 
directly, but it is accessible by using the so-called low-rank arithmetics (LRA). 
The equation is solved by matrix reformulation of the method of conjugate gradients (MCG). 
We are interested in the behavior of ranks of the solution approximations $X_\ell$, 
residuals $R_\ell$, and direction vectors $P_\ell$ during iterations.}

\keywords{conjugate gradients, Lyapunov equation, low-rank arithmetics} 

\MSC{15A06, 15A24, 65F10, 65F45, 65F55}

\begin{document}

\maketitle

\section{Introduction}
The urge to solve Lyapunov matrix equation comes from the control theory, 
specifically from the stability of linear dynamical systems, that are
represented by differential equations \cite{antoulas}. In ideal circumstances
the system is {\em controllable} and {\em observable}. For stable systems the 
controllability and observability are determined by corresponding grammians that 
satisfy the Lyapunov equation. 

For solving the Lyapunov equation we use the method of conjugate gradients (CG) 
in a special matrix form (MCG). First we will take a closer look at the structure 
of the equation and the possibility of using MCG. 

\section{The Lyapunov equation}\label{sec:Lyap_eq}

We consider the Lyapunov equation 
\begin{equation}\label{eq:lyap}
\mathscr{L}_A(X) = F, \qquad\text{where}\qquad \mathscr{L}_A(X) = AX + XA^T, \qquad A,X,F\in\mathbb{R}^{n\times n}.
\end{equation}
If $F$ is symmetric, $F^T=F$, then by transposing the whole equation we get
\[
\mathscr{L}_A(X^T) = AX^T + X^TA^T = (AX + XA^T)^T = F^T = F. 
\]
Moreover, if the Lyapunov equation (\ref{eq:lyap}) has the unique solution, then clearly $X^T=X$ has to be symmetric as well.

By vectorization of (\ref{eq:lyap}) and using Kronecker product notation we can 
rewrite the original equation in the form of standard linear system (see for example \cite{kpt})
\begin{equation}\label{eq:L=IA+AI}
L_Ax = f, \quad \text{where}\quad L_A = (A \otimes I_n + I_n \otimes A)\in\mathbb{R}^{n^2\times n^2}, \quad f = \vec(F)\in\mathbb{R}^{n^2}.
\end{equation}
Thus the unique solvability of (\ref{eq:lyap}) is equivalent to the regularity of $L_A$.

We further assume $A$ to be symmetric positive definite (SPD). Recall that for symmetric matrices, positive definiteness is equivalent
to positive eigenvalues,
\begin{equation}\label{spA}
	\sp(A) = \{\lambda_i(A) \;:\; i=1,\ldots,n\}.
\end{equation}
Then it is easy to see that Lyapunov matrix $L_A$ is symmetric as well, and also positive definite because (see for example \cite{kpt})
\begin{equation}\label{spLA}
	\sp(L_A) = \{\lambda_i(A) + \lambda_j(A) \;:\; i=1,\ldots,n, \, j=1,\ldots,n\},
\end{equation}
so $L_A$ is SPD and thus regular. This allows us to use CG (in fact MCG) for solving $L_Ax=f$ (in fact $AX+XA^T=F$, $x=\mathrm{vec}(X)$).  
If the right-hand side is of low rank, then the solution can be approximated by a low rank matrix; see \cite{kress}, \cite{Penzl}, or \cite{sabino}. 
This is important, because the prospective goal is to simulate the use of low-rank arithmetics (LRA), and also to use LRA in practical computation.

\section{Matrix formulation of the conjugate gradient method (MCG)}

The algorithm is fully analogous to standard CG, it is just slightly modified for solving our matrix equation;
see Algorithm \ref{alg:mcg}.
We replace all vectors by corresponding matrices, and matrix-vector product by application 
of the Lyapunov operator $\mathscr{L}_A$.

\begin{algorithm}[ht!]
\caption{MCG for solving matrix Lyapunov equation $AX+ XA^T = F$} \label{alg:mcg}
\begin{algorithmic}[1]
\STATE {\bf input} $A, F, X_0\in\mathbb{R}^{n\times n}$
\hfill\COMMENT{SPD matrix, right-hand side, init. est. (std. $X_0=0$)}
\STATE $R_0 \leftarrow F-\mathscr{L}_A(X_0)$ \label{line:init_res}
\hfill\COMMENT{initial residual}
\STATE $\rho_0 \leftarrow \langle R_0,R_0 \rangle$
\STATE $P_0 \leftarrow R_0$
\hfill\COMMENT{initial direction vector}
\FOR{$\ell = 1,2,\ldots$}
  \STATE $W_\ell \leftarrow \mathscr{L}_A(P_{\ell-1})$ \label{line:wk}
  \hfill\COMMENT{application of Lyapunov operator}
  \STATE $\alpha_\ell \leftarrow \rho_{\ell-1}/\langle W_\ell,P_{\ell-1} \rangle$
  \STATE $\beta_\ell \leftarrow \rho_{\ell-1}$
  \STATE $X_\ell \leftarrow X_{\ell-1} + P_{\ell-1} \alpha_\ell$ \label{line:sol_appr}
  \hfill\COMMENT{solution approximation}
  \STATE $R_\ell \leftarrow R_{\ell-1} - W_\ell \alpha_\ell$ \label{line:residual}
  \hfill\COMMENT{calculated residual}
  \STATE $\rho_\ell \leftarrow \langle R_\ell,R_\ell \rangle$
  \STATE $\beta_\ell \leftarrow \rho_\ell/\beta_\ell$
  \STATE $P_\ell \leftarrow R_\ell + P_{\ell-1} \beta_\ell$ \label{line:direc_vec}
  \hfill\COMMENT{direction vector} \mbox{}
\ENDFOR
\RETURN $X_\ell$
\end{algorithmic}
\end{algorithm}

The key issue of MCG is the propagation of rank, in particular how ranks of individual iterands $X_\ell$, $R_\ell$, $P_\ell$, and $W_\ell$ 
can grow with $\ell$. The implementation of LRA is then realized by working only with eigenvectors and nonzero eigenvalues of these
(symmetric) matrices. Concretely, $S=S^T$ is stored as $(V_S,\Lambda_S)$, where $S=V_S\Lambda_SV_S^T$ and ideally $V_S^TV_S=I$ and 
$\Lambda_S=\diag(\lambda_1(S),\lambda_2(S),\ldots)$.

\subsection{Propagation of rank in saxpy operations}\label{ssec:saxpy}

Saxpy operation is used four times, but only three times in every MCG iteration, specifically in lines~\ref{line:init_res} 
(where it is combined with application of Lyapunov operator $\mathscr{L}_A$; see 
Section~\ref{ssec:lyap}), \ref{line:sol_appr}, \ref{line:residual}, and \ref{line:direc_vec},
taking form of
\[
	S = M + N\omega, \qquad \text{where} \qquad S,M,N\in\mathbb{R}^{n\times n}, \quad\text{and}\quad \omega\in\mathbb{R}.
\]
Then symmetry of $M$ and $N$ implies symmetry of $S$; $M^T=M, N^T=N \Longrightarrow S^T=S$. Denote
$\varrho_M = \rank(M)$. 
Then the spectral decomposition of $M$ takes form
\[
	\begin{split}
	M = V_M\Lambda_M V_M^T, \qquad 
    & V_M\in\mathbb{R}^{n\times \varrho_M}, \qquad V_M^TV_M = I_{\varrho_M}, \\
	& \Lambda_M = \diag(\lambda_1(M),\lambda_2(M),\ldots,\lambda_{\varrho_M}(M)) \in\mathbb{R}^{\varrho_M\times \varrho_M}, \\
    & |\lambda_1(M)| \geq |\lambda_2(M)| \geq \cdots \geq |\lambda_{\varrho_M}(M)| > 0, 
	\end{split}
\]
similarly for $N$. 

The spectral decomposition of $S$ is obtained in four steps: {\em First} step is the formal saxpy  
\begin{equation}\label{eq:saxpy:0}
	S = V_M\Lambda_M V_M^T + V_N\Lambda_N V_N^T\omega 
	= [V_M,V_N] \left[\begin{array}{cc} \Lambda_M & 0 \\ 0 & \Lambda_N \omega \end{array}\right] [V_M,V_N]^T,
\end{equation}
yielding already a spectral-like decomposition. 
{\em Second} step is re-orthonormalization of $[V_M,V_N]$ by using thin QR decomposition
\begin{equation}\label{eq:saxpy:1}
	[V_M,V_N] = QR, \qquad Q\in\mathbb{R}^{n\times\varrho'}, \quad R\in\mathbb{R}^{\varrho'\times(\varrho_M+\varrho_N)}.
\end{equation}
{\em Third} step is spectral decomposing of {$\rho'$}-by-{$\rho'$} matrix
\begin{equation}\label{eq:saxpy:2}
	\left(R \left[\begin{array}{cc} \Lambda_M & 0 \\ 0 & \Lambda_N \omega \end{array}\right] R^T\right) = V\Lambda_S V^T, \qquad V\in\mathbb{R}^{\varrho'\times\varrho_S}, \quad \Lambda_S\in\mathbb{R}^{\varrho_S\times\varrho_S}.
\end{equation}
{\em Fourth} and final step is just a multiplication
\[
	S = (Q V)\Lambda_S (Q V)^T = V_S \Lambda_S V_S^T
\]
that results in the wanted decomposition of $S$.

Clearly, the rank of $S$ is bounded as follows
\[
	0 \leq |\varrho_M - \varrho_N| \leq 	\varrho_S \leq \varrho_M + \varrho_N \leq 2\max(\varrho_M,\varrho_N),
\]
i.e., in the worst case, the rank doubled in each of the lines \ref{line:sol_appr}, \ref{line:residual}, and \ref{line:direc_vec}.
The steps described by equations (\ref{eq:saxpy:1})--(\ref{eq:saxpy:2}) can be seen as a compression of expression (\ref{eq:saxpy:0}).
This compression point is the place, where we can control the rank growth in practical implementation (by rounding some eigenvalues
of $\Lambda_S$ to zero); see for example \cite{kpt}.

\subsection{Propagation of rank in application of Lyapunov operator $\mathscr{L}_A$}\label{ssec:lyap}

The Lyapunov operator $\mathscr{L}_A$ behaves similarly to the saxpy operation in terms of rank growth. 
It appears twice in the algorithm, but only once in every MCG iteration. Specifically in calculating 
initial residual $R_0$ in line \ref{line:init_res} (where it is followed by saxpy operation; see Section~\ref{ssec:saxpy}), 
and then in line \ref{line:wk}. It takes form of
\[
	S = \mathscr{L}_A(M) = AM + MA^T, \qquad \text{where} \qquad A,S,M,\in\mathbb{R}^{n\times n}.
\]
Again, symmetry of $M$ implies symmetry of $S$; $M^T=M\Longrightarrow S^T=S$.
Using the spectral decomposition of $M$,
\begin{equation}\label{eq:lyap_oper_1}
	\begin{split}
	S = AM + MA^T &= A V_M\Lambda_M V_M^T + V_M\Lambda_M V_M^T A^T \\
	 &= [A,V_M] \left[\begin{array}{cc} \Lambda_M & 0 \\ 0 & \Lambda_M \end{array}\right] [A,V_M]^T.
\end{split}
\end{equation}
This spectral-like decomposition of $S$ needs to be compressed. The compression stage is 
fully analogous to the saxpy case in Section \ref{ssec:saxpy}. Again the rank can be in 
the worst case doubled in line \ref{line:wk}. The compression is again the place where 
we can control the rank growth. 

\subsection{Experiment}\label{ssec:exp}

Altogether, Algorithm \ref{alg:mcg} with $X_0^T=X_0$ and $F^T=F$ preserves symmetry.
But it can suffer from four rank doublings per MCG iteration. 
Let us explore this behavior on a simple experiment. We used the discretization 
of the 1D Laplace operator with a Dirichlet boundary conditions to fulfill the role 
of the matrix $A$ (up to the sign), so it is SPD 
and sparse. Matrix $F$ is of ones, so symmetric and low-rank ($\varrho_F=1$),
\[ 
  A = \left[
	\begin{array}{ccccc}
		2 & -1 & & \color{gray}0 \\
		-1 & \ddots & \ddots & \\
		& \ddots & \ddots & -1 \\
		\color{gray}0 & & -1 & 2
	\end{array}
	\right], 
    \quad 
	F = \left[
	\begin{array}{ccc}
	1 & \cdots & 1 \\
	\vdots & \ddots & \vdots \\
	1 & \cdots & 1
\end{array}\right]
 = bb^T ,
    \quad 
b= \!\left[\begin{array}{c}
	1 \\ \vdots \\	1
\end{array}\right]\in\mathbb{R}^{100}.
\]
Thus the Lyapunov operator $\mathscr{L}_A$ represented by matrix $L_A$ is the discretization of 
the 2D Laplace operator with Dirichlet boundary conditions (up to the sign).

Now we inspect the numerical rank of the calculated solution approximations $X_\ell$; see Figure \ref{fig:eigs_xk}. 
Plot in this figure shows evolution of eigenvalues of matrices $X_\ell$ during MCG iterations;
numerical rank is emphasized by the black line. In the beginning we observe the rapid growth of the
rank (related to the potential rank doubling), which in principle can not continue for too long. 
Then we see, that the growth stopped and the rank stagnated (at rank $\sim50$) for significant 
number of iterations (roughly in iterations $100$--$200$). Surprisingly it stopped before reaching 
the dimension $n=100$. Then the rank decreases again, while absolute values of $\lambda_j(X_\ell)$ finally 
resemble an exponential decay; see for example \cite{kress}, \cite{Penzl}, or \cite{sabino}.

\begin{figure}[t!]
\begin{center}
\includegraphics[width=.8\textwidth]{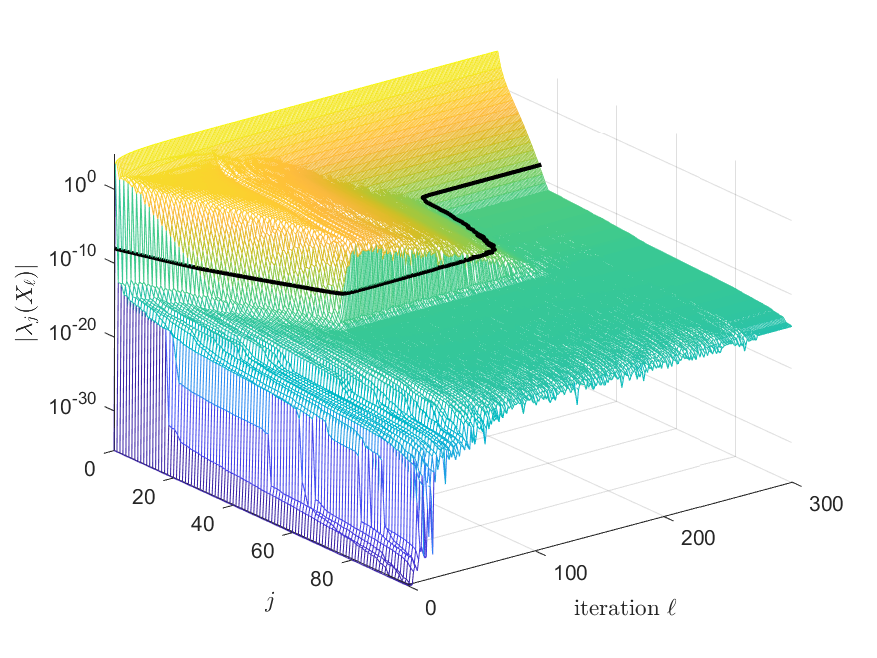}
\caption{Absolute values of $\lambda_j(X_\ell)$, $j=1,\ldots,100$, for $\ell=1,\ldots,300$. 
The black curve shows the progression of the numerical rank of $X_\ell$ measured as number of eigenvalues with absolute value greater 
than $\epsilon|\lambda_1(X_\ell)|n$, where $\epsilon$ is the machine precision.
The red line highlights the exponential decay of $|\lambda_j(X_{300})|$.}\label{fig:eigs_xk}
\begin{picture}(0,0)
    \put(0,-40){
        \put(22  ,372  ){\rotatebox{-5}{\color{red}\line(1,-2){1cm}}}
        \put(22.2,372.2){\rotatebox{-5}{\color{red}\line(1,-2){1cm}}}
        \put(22.4,372.4){\rotatebox{-5}{\color{red}\line(1,-2){1cm}}}
        \put(22.6,372.6){\rotatebox{-5}{\color{red}\line(1,-2){1cm}}}
        \put(22.8,372.8){\rotatebox{-5}{\color{red}\line(1,-2){1cm}}}
        \put(23  ,373  ){\rotatebox{-5}{\color{red}\line(1,-2){1cm}}}
    }
\end{picture}
\end{center}
\end{figure}
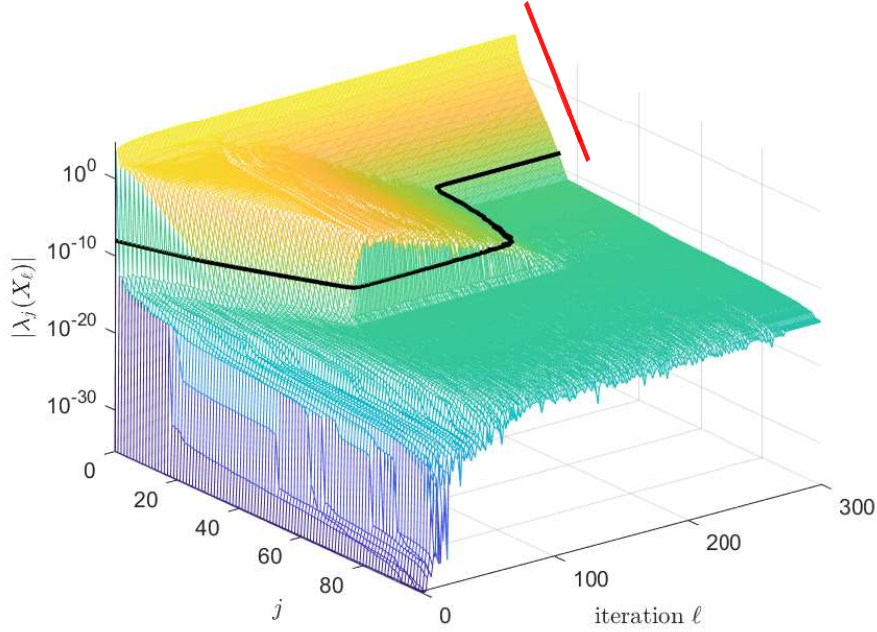

This intermediate rank-growth is limiting for usage of low-rank arithmetics. We can fight this using
preconditioners, however, this phenomenon is not fully understood yet; see, e.g., \cite{kpt}, \cite{eppler}, \cite{Simoncini}.
In the rest of this paper, we focus on answering the question: {\em Why does the rank stop to grow, while not reaching dimension $n=100$?}

\section{Krylov subspaces with Lyapunov operator}

It is well known, that the standard CG for solving equation $Ax = b$ with $x_0 = 0$ 
is closely related to the Krylov subspaces defined as
\[
\mathcal{K}_\ell(A,b) = \spn (\{b,Ab,A^2b,\ldots, A^{\ell-1}b\}).
\]
In particular
\begin{equation}\label{eq:wl_xl_rl_pl}
x_\ell \in \mathcal{K}_\ell(A,b), \;\;
r_\ell \in \mathcal{K}_{\ell+1}(A,b), \;\;
p_\ell \in \mathcal{K}_{\ell+1}(A,b), \;\;
w_\ell = Ap_{\ell-1} \in A\mathcal{K}_\ell(A,b);
\end{equation}
see \cite{hesteness}, or \cite{LiesenStrakos}.

\subsection{Relations between Krylov subspaces generated by $L_A$ and $\mathscr{L}_A$}

For the standard CG applied on $L_Ax=f$, where $L_A=A\otimes I_n+I_n\otimes A$, $f=\vec(F)$, see (\ref{eq:L=IA+AI}), 
the corresponding $\ell$th Krylov subspace takes form
\[
\mathcal{K}_\ell(L_A,f) = \spn\left( \{f,L_Af,L_A^2f, \ldots, L_A^{\ell-1}f\}\right) \subseteq \mathbb{R}^{n^2}.
\]
Since it represents vectorization of our MCG (so MCG is just matricized form of 
the standard CG), then the $\ell$th Krylov subspace in the MCG takes form
\[
\mathcal{K}_\ell(\mathscr{L}_A,F) = \spn\left( \{F,\mathscr{L}_A(F),\mathscr{L}_A^2(F), 
		\ldots, \mathscr{L}_A^{\ell-1}(F)\}\right) \subseteq \mathbb{R}^{n\times n},
\]
where 
\[
\mathscr{L}_A^\ell(F) = \mathscr{L}_A(\mathscr{L}_A^{\ell-1}(F))\qquad \text{and}\qquad \mathscr{L}_A^0(F) = F.
\]
It is useful to explicitly write down relations between these operators and subspaces.

\begin{lemma}\label{lem:vecL}
Let 
$\mathscr{L}_A$ be the Lyapunov operator defined in (\ref{eq:lyap}),
$L_A$ be the matrix defined in (\ref{eq:L=IA+AI})
$A \in\mathbb{R}^{n\times n}$ be an SPD matrix, 
and $F\in\mathbb{R}^{n\times n}$, $f=\vec(F)$.
Then
\begin{eqnarray}
\vec(\mathscr{L}_A^\ell(F)) &\!\!=\!\!& L_A^\ell f, \\
\vec(\mathcal{K}_\ell(\mathscr{L}_A,F)) &\!\!=\!\!& \mathcal{K}_\ell(L_A,f), \label{eq:lem1eq2}
\end{eqnarray}
where the vectorization on the left side of the latter equality is applied on every element of the set.
\end{lemma}

\begin{proof} From (\ref{eq:lyap}) and (\ref{eq:L=IA+AI}) we have  $\vec(\mathscr{L}_A(X)) = \vec(F) = f = L_A\vec(X)$.
Thus for $X=F$, $\vec(\mathscr{L}_A(F)) = L_A f$. The rest is induction with $X=\mathscr{L}_A^{\ell-1}(F)$,
\[
    \vec(\mathscr{L}_A^\ell(F)) = \vec(\mathscr{L}_A(\mathscr{L}_A^{\ell-1}(F))) = L_A\vec(\mathscr{L}_A^{\ell-1}(F)) =  L_A(L_A^{\ell-1}f) = L_A^\ell f.
\]
The latter equality is then the consequence of linearity of vectorization.
\end{proof}

\subsection{Relations between Krylov subspaces generated by $\mathscr{L}_A$ and $A$}

Let us consider two corresponding elements in both Krylov spaces
\[
V \in \mathcal{K}_\ell(\mathscr{L}_A,F) \subseteq \mathbb{R}^{n\times n} \quad \longleftrightarrow \quad v \in \mathcal{K}_\ell(L_A,\vec(F)) \in \mathbb{R}^{n^2},
\]
the relation between them is given by the equality $\vec(V) = v$; see Lemma \ref{lem:vecL}. 
Consider a mapping that selects the $j$th column from a matrix,
\[\begin{split}
\phi_j : \mathbb{R}^{n\times n} &\longrightarrow \mathbb{R}^n  \qquad \text{so that}\qquad \phi_j(V) = Ve_j^{(n)},
\end{split}\]
where $e_j^{(n)}$ is the $j$th column of the identity matrix $I_n$. Corresponding mapping in the other space is 
\[\begin{split}
\phi_j : \mathbb{R}^{n^2} &\longrightarrow \mathbb{R}^n
  \qquad \text{so that}\qquad \phi_j(v) = \left[e_{(j-1)n+1}^{(n^2)}, \ldots, e_{jn}^{(n^2)}\right]^Tv.
\end{split}\]
We use the same symbol for both mappings. They are defined on different spaces but give 
the same result 
\begin{equation}\label{eq:phirel}
    \phi_j(V) = \phi_j(\vec(V)).
\end{equation}
Now we are ready to explore the relation between $\mathcal{K}_\ell(\mathscr{L}_A,F)$ and $\mathcal{K}_\ell(A,b)$.

\begin{theorem}\label{th:phiK}
Let 
$\mathscr{L}_A$ be the Lyapunov operator defined in (\ref{eq:lyap}),
$A \in\mathbb{R}^{n\times n}$ be an SPD matrix, 
and $F=bb^T\in\mathbb{R}^{n\times n}$.
Then the following applies
\begin{equation}\label{eq:mapping}
\phi_j \big(\mathcal{K}_\ell(\mathscr{L}_A,F)\big) \subseteq \mathcal{K}_\ell(A,b),
\end{equation}
where on the left side we apply the mapping $\phi_j$ on every element of the set.
\end{theorem}

\begin{proof}
For $\ell=0$ the assertion is rather trivial,
\[
    \mathcal{K}_0(\mathscr{L}_A,F) = \spn\left( \{F\}\right), \qquad \text{thus} \qquad
    \phi_j\big(\mathcal{K}_0(\mathscr{L}_A,F)\big) = \phi_j\big(\spn(\{F\})\big).
\]
Recall that $F=bb^T$, $b=[\beta_1,\ldots,\beta_n]^T$. Since $\phi_j$ is linear mapping, then
\[
\phi_j\big(\spn(\{F\})\big) = \spn\big(\{\phi_j(F)\}\big) = \spn\big(\{b \beta_j\}\big) = \spn\big(\{b\}\big)\beta_j. 
\]
Finally
\[
    \spn\big(\{b\}\big)\beta_j = \left\{
    \begin{array}{ll}
    \{0\} & \text{(for $\beta_j=0$)} \\
    \spn(\{b\}) & \text{(for $\beta_j\neq0$)}
    \end{array}\right\}\subseteq \spn(\{b\}) = \mathcal{K}_0(A,b).
\]
Thus $\phi_j\big(\mathcal{K}_0(\mathscr{L}_A,F)\big) \subseteq \mathcal{K}_0(A,b)$. 

For clarity we prove it also for $\ell=1$. Using linearity of mapping $\phi_j$,
\[
\phi_j(\mathcal{K}_1(\mathscr{L}_A,F)) 
     = \phi_j(\spn (\{F, \mathscr{L}_A(F)\})) 
     = \spn( \{\phi_j(F), \phi_j(\mathscr{L}_A(F))\}). 
\]
We already know (from the $\ell=0$ case) that $\phi_j(F) \in \mathcal{K}_0(A,b) \subseteq \mathcal{K}_1(A,b)$.
Now we show it for the other vector
\[\begin{split}
     \phi_j(\mathscr{L}_A(F))
     &= \phi_j(AF) + \phi_j(FA^T) 
     = A\phi_j(F) + F\phi_j(A^T) \\ 
     &= A\phi_j(F) + {\textstyle\sum_{k=1}^n\nolimits}{\phi_k(F) a_{j,k}} 
     = A(b\beta_j) + {\textstyle\sum_{k=1}^n\nolimits}{(b\beta_k) a_{j,k}} \\
     &= (Ab)\beta_j + b \left({\textstyle\sum_{k=1}^n\nolimits}{\beta_k a_{j,k}}\right) \in \mathcal{K}_1(A,b).
\end{split}\]
Since $\phi_j(F)$, $\phi_j(\mathscr{L}_A(F))\in\mathcal{K}_1(A,b)$, thus $\phi_j\big(\mathcal{K}_1(\mathscr{L}_A,F)\big) \subseteq \mathcal{K}_1(A,b)$. 

Now we are ready for induction. Assume that 
\[
\phi_j(\mathcal{K}_{\ell-1}(\mathscr{L}_A,F))\subseteq \mathcal{K}_{\ell-1}(A,b).
\]
Analogously to the $\ell=1$ case,
\[\begin{split}
\phi_j(\mathcal{K}_\ell(\mathscr{L}_A,F)) &= \phi_j(\spn(\{F,\mathscr{L}_A(F), \ldots, \mathscr{L}_A^{\ell-2}(F), \mathscr{L}_A^{\ell-1}(F)\})) \\
 &= \spn(\{\phi_j(F),\phi_j(\mathscr{L}_A(F)), \ldots, \phi_j(\mathscr{L}_A^{\ell-2}(F)), \phi_j(\mathscr{L}_A^{\ell-1}(F))\}). \\
\end{split}\]
By the induction hypothesis
$\phi_j(F),\ldots,\phi_j(\mathscr{L}_A^{\ell-2}(F)) \in \mathcal{K}_{\ell-1}(A,b) \subseteq \mathcal{K}_\ell(A,b)$.
Therefore we only have to elaborate on the last vector 
\[
    \phi_j(\mathscr{L}_A^{\ell-1}(F)) = \phi_j(\mathscr{L}_A(\mathscr{L}_A^{\ell-2}(F))) 
     = A\phi_j(\mathscr{L}_A^{\ell-2}(F)) + {\textstyle\sum_{k = 1}^{n}\nolimits} {\phi_k(\mathscr{L}_A^{\ell-2}(F))a_{j,k}}. 
\]
Using the induction hypothesis once more 
\[
    A\phi_j(\mathscr{L}_A^{\ell-2}(F))\in A\mathcal{K}_{\ell-1}(A,b) \subseteq \mathcal{K}_{\ell}(A,b), \qquad
    \phi_k(\mathscr{L}_A^{\ell-2}(F))\in \mathcal{K}_{\ell-1}(A,b), 
\]
thus $\phi_j\big(\mathcal{K}_\ell(\mathscr{L}_A,F)\big) \subseteq \mathcal{K}_\ell(A,b)$. 
\end{proof}

\subsection{Krylov bound for ranks of iterands}

Recall the relations for the iterands $x_\ell$, $r_\ell$, $p_\ell$, and $w_\ell$ 
in the classical CG; see (\ref{eq:wl_xl_rl_pl}). In the case of MCG we 
immediately get
\[
    X_\ell\in\mathcal{K}_\ell(\mathscr{L}_A,F),\;\; 
    R_\ell\in\mathcal{K}_{\ell+1}(\mathscr{L}_A,F),\;\; 
    P_\ell\in\mathcal{K}_{\ell+1}(\mathscr{L}_A,F),\;\;
    W_\ell\in A\mathcal{K}_\ell(\mathscr{L}_A,F).
\]
Theorem \ref{th:phiK} has the following straightforward corollary.

\begin{corollary}\label{cor:main}
Let 
$\mathscr{L}_A$ be the Lyapunov operator defined in (\ref{eq:lyap}),
$A \in\mathbb{R}^{n\times n}$ be an SPD matrix, 
and $F=bb^T\in\mathbb{R}^{n\times n}$.
Consider the MCG applied on $\mathscr{L}_A(X)=F$ with the initial estimate $X_0=0$ 
and iterands $X_\ell$, $R_\ell$, $P_\ell$, and $W_\ell$; see Algorithm \ref{alg:mcg}.
Then
\begin{equation}\label{eq:kspaces}
    \begin{array}{rclcrcl}
    \phi_j (X_\ell) &\!\!\!\in\!\!\!& \mathcal{K}_\ell(A,b), &\quad&
    \phi_j (R_\ell) &\!\!\!\in\!\!\!& \mathcal{K}_{\ell+1}(A,b), \\
    \phi_j (P_\ell) &\!\!\!\in\!\!\!& \mathcal{K}_{\ell+1}(A,b), &\quad&
    \phi_j (W_\ell) &\!\!\!\in\!\!\!& A\mathcal{K}_\ell(A,b). \\
    \end{array}
\end{equation}
This further bounds the ranks of individual iterands
\begin{equation}\label{eq:dim}
    \begin{array}{rclcrcl}
    \rank(X_\ell) &\!\!\!\leq\!\!\!& \dim(\mathcal{K}_\ell(A,b)), &\quad&
    \rank(R_\ell) &\!\!\!\leq\!\!\!& \dim(\mathcal{K}_{\ell+1}(A,b)), \\
    \rank(P_\ell) &\!\!\!\leq\!\!\!& \dim(\mathcal{K}_{\ell+1}(A,b)), &\quad&
    \rank(W_\ell) &\!\!\!\leq\!\!\!& \dim(\mathcal{K}_\ell(A,b)). \\
    \end{array}
\end{equation}
\end{corollary}

\section{Maximal number of iterations and maximal ranks of iterands in MCG}

Consider for a moment classical CG applied on $Ax=b$, where $A$ is SPD matrix. Denote
\begin{equation}\label{kapA}
	\kappa_A = |\sp(A)|
\end{equation}
the number of distinct eigenvalues of $A$. Consider also the spectral decomposition
\[
    A = V_A\Lambda_A V_A^T, \qquad V_A=[V_1,V_2,\dots,V_{\kappa_A}], \qquad V_A^{-1}=V_A^T,
\]
where columns of $V_s$ form an orthonormal basis of eigenspace corresponding to the $s$th 
largest eigenvalue, $s=1,2,\ldots,\kappa_A$. Further denote 
\begin{equation}\label{kapAb}
	\kappa_{A,b} = \mathrm{nnz}([\|V_1^Tb\|,\|V_2^Tb\|,\ldots,\|V_{\kappa_A}^Tb\|]),
\end{equation}
the number of nonzero components of $b$ in eigenspaces of $A$. Clearly $\kappa_{A,b}\leq \kappa_{A}$.
It is well known that (in exact arithmetics):
\begin{itemize}\itemsep 0pt
\item If $\mathcal{K}_\ell(A,b)=\mathcal{K}_{\ell+1}(A,b)$, then $\dim(\mathcal{K}_\ell(A,b))=\kappa_{A,b}$ and
\item CG applied on $Ax=b$ with $x_0=0$ converges in the $\kappa_{A,b}$th iteration; 
\end{itemize}
see for example \cite{LiesenStrakos}.

\subsection{Maximal number of iterations of MCG}

To get the exact number of iterations of MCG applied on $\mathscr{L}_A(X)=F$, $F=bb^T$, with $X_0=0$,
we need to reveal the value of $\kappa_{L_A,f}$, $f=\vec(F)$. Since it is bounded from above
by the number of distinct eigenvalues of $L_A$, we can immediately see, using (\ref{spLA}), that
\[
    \kappa_{L_A} = |\sp(L_A)| \leq \frac{\kappa_A(\kappa_A+1)}2.
\] 
Note that the eigenvalues of $L_A$ cannot be all simple (unless $n=1$) due to the commutativity of $+$.
Also note that the inequality there could be strict, e.g., for $\sp(A)=\{1,2,3\}$, we get $\kappa_A = 3$, 
$\kappa_{L_A} = 5$. So the number of distinct eigenvalues of $A$ only bounds the number of distinct eigenvalues
of $L_A$. Further multiplicities in the spectrum of $L_A$ could appear due to particular values of $\lambda_i(A)$s. 
 
Similarly, using suitable choices of bases of $V_s$, $s=1,2,\ldots,\kappa_A$, it can be shown that
\[
    \kappa_{L_A,f} \leq \frac{\kappa_{A,b}(\kappa_{A,b}+1)}2;
\] 
see for example \cite{LungovaDP}. 

 
\subsection{Maximal ranks of iterands in MCG}

The behavior of classical CG together with Corollary \ref{cor:main} answers how much
the ranks of individual iterands can grow. 

\begin{corollary}\label{cor:main3}
Let 
$\mathscr{L}_A$ be the Lyapunov operator defined in (\ref{eq:lyap}),
$A \in\mathbb{R}^{n\times n}$ be an SPD matrix, 
and $F=bb^T\in\mathbb{R}^{n\times n}$.
Consider the MCG applied on $\mathscr{L}_A(X)=F$ with the initial estimate $X_0=0$ 
and iterands $X_\ell$, $R_\ell$, $P_\ell$, and $W_\ell$; see Algorithm \ref{alg:mcg}.
Then
\begin{equation}\label{eq:dim2}
    \rank(X_\ell) \leq \kappa_{A,b}, \; 
    \rank(R_\ell) \leq \kappa_{A,b}, \; 
    \rank(P_\ell) \leq \kappa_{A,b}, \; 
    \rank(W_\ell) \leq \kappa_{A,b},
\end{equation}
where $\kappa_{A,b}$ is given in (\ref{kapAb}).
\end{corollary}

\subsection{Back to the experiment}

Recall our experiment in Section \ref{ssec:exp}.
The given matrix $A$ (discretization of 1D Laplace operator with Dirichlet boundary condition) 
has the following eigenvalues (all simple) and eigenvectors
\[
    \lambda_j = 2-2\cos\left(\frac{j\pi}{n+1}\right), \quad v_j = \left[ \sin\left(\frac{t j\pi}{n+1}\right)\right]_{t=1}^n, \quad j=1,\ldots,n,
\]
thus $\kappa_A = n = 100$, $\kappa_{L_A}\leq \frac{n(n+1)}{2}=5050$. Since $F=bb^T$ where
$b=[1,1,\ldots,1]^T$, then the inner-products
\[
    \langle b,v_j\rangle = \sum_{t=1}^{n}{\sin\left(\frac{t j\pi}{n+1}\right)} 
    \left\{
    \begin{array}{ll}
    \neq 0 & \text{for $j=1,3,5,\ldots$ (odd)} \\
    = 0 & \text{for $j=2,4,6,\ldots$ (even)} \\
    \end{array}\right..
\]
and $\kappa_{A,b}=\lfloor\frac{n}2\rfloor=50$, $\kappa_{L_A,f}\leq\frac{n(n+2)}8=1275$. 

In particular, we see the effect of bound $\rank(X_\ell)\leq \kappa_{A,b}=50$ in Figure \ref{fig:eigs_xk}
between iterations $100$--$200$. The rank here stagnates and cannot grow over $50$.

\section{Conclusions}

We have studied the MCG method applied on Lyapunov equation $\mathscr{L}_A(X)=bb^T$ with a rank-one right-hand side. 
We have presented an important theorem with several corollaries, while establishing bounds for ranks of all MCG iterands. 
These ranks are controlled by the spectral properties of matrix $A$, while MCG convergence by properties of operator 
$\mathscr{L}_A$. An experiment shows that our bound is tight and the theoretical maximal allowed rank can also be 
attained for a reasonable problem.

\section*{Acknowledgements}
This work was supported by internal grant SGS-2026-4602 of TU Liberec.


\begin{thebibliography}{13}

\bibitem{antoulas}
Antoulas, A. C.:
\emph{Approximation of Large-Scale Dynamical Systems},
Advances in Design and Control 6.
SIAM Publishing, Philadelphia, 2005.

\bibitem{eppler}
Bollh\"{o}fer, M. and Eppler, A. K.:
Low-rank Cholesky factor Krylov subspace methods for generalized projected Lyapunov equations.
In: P. Benner (Ed.) \emph{System Reduction for Nanoscale IC Design}, Mathematics in Industry 20, pp. 157--193.
Springer, Cham 2017.


\bibitem{hesteness}
Hestenes, M. R. and Stiefel, E.:
Methods of conjugate gradients for solving linear systems. 
J. Res. Natl. Bur. Stand. {\bf 49}(6) (1952), 409--436.

\bibitem{kpt}
Kressner, D., Ple\v{s}inger, M., and Tobler, C.:
A preconditioned low-rank CG method for parameter-dependent Lyapunov matrix equations.
Numer. Linear Algebra Appl. {\bf 21}(5) (2014), 666--684.

\bibitem{kress}
Kressner, D. and Tobler, C.:
Krylov subspace methods for linear systems with tensor product structure. 
SIAM J. Matrix Anal. Appl. {\bf 31}(4) (2010), 1688--1714.

\bibitem{LiesenStrakos}
Liesen, J. and Strako\v{s}, Z.:
\emph{Krylov Subspace Methods},
Numerical Mathematics and Scientific Computations. 
Oxford University Press, Oxford, 2013.


\bibitem{LungovaDP}
Lungov\'{a}, J.:
\emph{Selected properties of the conjugate gradient method when applied on a matrix Lyapunov equation}.
Diploma thesis, TU Liberec, Liberec, 2025.

\bibitem{Penzl}
Penzl, T.:
Eigenvalue decay bounds for solutions of Lyapunov equations: The symmetric case.
Syst Control Lett. {\bf 40}(2) (2000), 139--144.

\bibitem{sabino}
Sabino, J.:
\emph{Solution of large-scale Lyapunov equations via the block modified smith methods}. 
Ph.D. Thesis, Department of Computational and Applied Mathematics, Rice University, Houston, 2006


\bibitem{Simoncini}
Simoncini, V. and Hao, Y.:
Analysis of the truncated conjugate gradient method for linear matrix equations.
SIAM J. Matrix Anal. Appl. {\bf 44}(1) (2023), 359--381.


\end{thebibliography}
\end{document}